\documentclass[letterpaper, 10 pt, conference]{ieeeconf}  

\IEEEoverridecommandlockouts                              
\usepackage{graphicx}
\usepackage{amsmath}
\usepackage{amssymb}
\usepackage{CJK}
\usepackage{color}
\usepackage{subfigure}
\usepackage{dsfont}
\usepackage{subfigure}
\usepackage{booktabs}
\usepackage{microtype}
\usepackage{mathrsfs}
\usepackage{bm}
\usepackage{doi}
\usepackage{hyperref}
\usepackage{cite}
\usepackage{url}

\newtheorem{theorem}{Theorem}[section]

\newtheorem{definition}{Definition}
\newtheorem{assumption}{Assumption}

\newtheorem{lemma}[theorem]{Lemma}
\usepackage{algorithm,algorithmicx,algpseudocode}
\allowdisplaybreaks[4]

\definecolor{steelblue}{RGB}{70,130,180}
\definecolor{warmred}{RGB}{244, 104, 65}

\title{Solving Optimal Power Flow for Distribution Networks with State Estimation Feedback}

\author{Yi Guo, Xinyang Zhou, Changhong Zhao, Yue Chen, Tyler Summers, and Lijun Chen
\thanks{This work was authored in part by the National Renewable Energy Laboratory, operated by Alliance for Sustainable Energy, LLC, for the U.S. Department of Energy (DOE) under Contract No. DE-EE-0007998. Funding provided by U.S. Department of Energy Office of Energy Efficiency and Renewable Energy Solar Energy Technologies Office. The views expressed in the article do not necessarily represent the views of the DOE or the U.S. Government. The U.S. Government retains and the publisher, by accepting the article for publication, acknowledges that the U.S. Government retains a nonexclusive, paid-up, irrevocable, worldwide license to
publish or reproduce the published form of this work, or allow others to do so, for U.S. Government purposes. This matertial is also based on work supported by the National Science Foundation under grant CMMI-1728605.}
\thanks{Y. Guo and T. Summers are with the Department
of Mechanical Engineering, The University of Texas at Dallas, Richardson,
TX, USA, email: \{yi.guo2, tyler.summers\}@utdallas.edu.}
\thanks{X. Zhou and Y. Chen are with Department of Power System Engineering, National Renewable Energy Laboratory, Golden, CO, 80401, USA, email: \{xinyang.zhou, yue.chen\}@nrel.gov.}
\thanks{C. Zhao is with Department of Information Engineering, The Chinese University of Hong Kong, HK, email: chzhao@ie.cuhk.edu.hk.}
\thanks{L. Chen is with College of Engineering and Applied Science, The University of Colorado, Boulder, CO, 80309, USA, email: lijun.chen@colorado.edu.}
}
\begin{document}
\maketitle
\thispagestyle{empty}
\pagestyle{empty}

\begin{abstract}
Conventional optimal power flow (OPF) solvers assume full observability of the involved system states. However in practice, there is a lack of reliable system monitoring devices in the distribution networks. To close the gap between the theoretic algorithm design and practical implementation, this work proposes to solve the OPF problems based on the state estimation (SE) feedback for the distribution networks where only a part of the involved system states are physically measured. 
The SE feedback increases the observability of the under-measured system and provides more accurate system states monitoring when the measurements are noisy. We analytically investigate the convergence of the proposed algorithm. The numerical results demonstrate that the proposed approach is more robust to large pseudo measurement variability and inherent sensor noise in comparison to the other frameworks without SE feedback.
\end{abstract}

\section{Introduction}



Distribution networks are being instrumented with new technologies to replace traditional one-way distribution management systems and provide more flexibility to better accommodate a large portion of distributed energy resources (DERs). As the incentive-based control strategies dominate
the electricity markets, many of the customers in distribution networks become active and motivated end-users to optimize their own power usage. System operators must deal with significant variability from renewable energy resources and heterogeneous customer behaviors. In order to achieve an accurate and reliable real-time distribution network operation, future distribution networks will require a more sophisticated control scheme for an efficient and safe operation.

For this purpose, optimal power flow (OPF) problems are formulated to determine optimal policies for controllable devices to optimize various objectives and subject to the constraints at network and device levels. Many issues in distribution networks have been proposed and well studied recently through mathematical analysis and computational techniques \cite{mohsenian2010autonomous,callaway2010achieving,camacho2011control,primadianto2016review,ahmad2018distribution,maharjan2013dependable,bernstein2018load,zhou2017incentive,dall2016optimal,summers2015stochastic,zhang2018joint,chiang1990existence,piagi2006autonomous,guo2018data1,guo2018data2,palensky2011demand,dehghanpour2019game,liu2012trade,dehghanpour2018survey,gomez2004power,wu1990power,colombino2019online,zamzam2019data,zhou2019accelerated}. The flexibility of controllable devices, including distributed renewable energy resources (RESs) \cite{guo2018data2}, microgrids\cite{piagi2006autonomous}, electric vehicles \cite{zhang2018joint} and demand responsive loads \cite{mohsenian2010autonomous,callaway2010achieving,palensky2011demand}, can be used to promote the network performance (e.g., stability, reliability, efficiency and social-welfare \cite{zhou2017incentive,maharjan2013dependable}), which offer an optimal operation schedule for these devices based on their prescribed objective functions. 

In OPF problems for distribution networks, most frameworks assume complete availability of network states to compute controllable device online updates. However, in practice such assumptions are unjustifiable due to increasingly complex, extremely large-scale distribution networks with nonlinear time-varying nature, limited communication bendwidth, etc. To solve these issues, power system state estimators have been utilized through the Supervisory Control and Data Acquisition (SCADA), phasor measurement, topology processor and pseudo measurement to obtain a clearer picture of network states \cite{primadianto2016review, ahmad2018distribution,dehghanpour2019game,dehghanpour2018survey,liu2012trade,zamzam2019data}, which enhances the accuracy of system operations. Overall, all these directions have explored efficient control or monitoring methods for optimal power flows and state estimation problems, respectively. In fact, none of the existing OPF works explicitly take advantage of the state estimation information for an optimal schedule with limited knowledge of distribution networks. Even with sophisticated modeling and programming techniques, lack of timely update information within the system control phase can cause possible infeasibility or system collapse.

In this work, we first formulate a general convex OPF problem subject to linearized power flow equations and network-wise coupling constraints. We then apply well-established model-based feedback method to approximately solve the OPF problem through primal-dual gradient algorithm with measurement feedback from nonlinear power flow. In this way, we reduce modeling errors introduced by power flow linearization, while keep the algorithm scalable and computationally tractable \cite{dall2016optimal,colombino2019online,zhou2017incentive}. However, in practice there may be only a limited number of measuring devices deployed in distribution networks, rendering such measurement-based feedback unrealistic. To bridge the gap between the theoretical studies of existing OPF problems and further their practical implementation on distribution networks,
we therefore replace the measurement-based feedback with an estimated one based on the results of solving a well-formulated state estimation problem under appropriate assumptions that guarantee full observability \cite{gomez2004power,wu1990power}. Convergence of the gradient algorithm with state estimation feedback is analytically established and  numerically corroborated. Numerical results based on an optimal voltage regulation problem also illustrate that SE-based feedback achieves better voltage profile than measurement-based feedback by increasing observability of the network.

The rest of this paper is organized as follows. Section~\ref{sec:model} models the distribution systems. Section~\ref{sec:opf} formulates an OPF problem and introduces gradient algorithm with measurement feedback for solving it. Section~\ref{sec:se} formulates state estimation problem and design a realistic algorithm to add onto the original measurement feedback. Section~\ref{sec:num} demonstrates numerical results and Section~\ref{sec:con} concludes this paper.

 
\section{System Modeling}\label{sec:model}
Consider a distribution network denoted by a directed and connected graph $(\mathcal{N}_0,\mathcal{E})$, where $\mathcal{N}_0:= \mathcal{N}\cup\{0\}$ is a set of all ``buses" or ``nodes" with substation node 0 and $\mathcal{N}:= \{1,\dots,N\}$, and $\mathcal{E} \subset \mathcal{N}\times\mathcal{N}$ is a set of ``links" or ``lines" for all $(i,j) \in \mathcal{E}$. Let $V_i :=|V_i|e^{j\angle V_i} \in \mathbf{C}$ and $I_i:= |I_i|e^{j\angle I_i}\in \mathbf{C}$ denote the phasor for the line-to-ground voltage and the current injection at node $i\in\mathcal{N}$. The absolute values $|V_i|$ and $|I_i|$ denote the signal root-mean-square values and the phase $\angle V_i$ and $\angle I_i$ corresponding to the phase angle with respect to the global reference. We collect these variables into complex vectors $\mathbf{v}:=[V_1,V_2,\dots,V_N]^\intercal \in\mathbf{C}^N$ and $\mathbf{i}:=[I_1,I_2,\dots,I_N]^\intercal \in \mathbf{C}^N$. We denote the complex admittance of line $(i,j)\in\mathcal{E}$ by $y_{ij} \in \mathbf{C}$. The admittance matrix $\mathbf{Y}\in \mathbf{C}^{N\times N}$ is given by
\begin{equation}\label{admittanceMatrix}
Y_{ij} = \left\{ \begin{array}{ll}
\sum_{l \sim i} y_{il} + y_{ii}, & \textrm{if $i=j$}\\
-y_{ij}, & (i,j) \in \mathcal{E} \\
0, &  (i,j) \notin \mathcal{E}
\end{array} \right .,
\end{equation}
where $l \sim i$ indicates the connection between node $l$ and node $i$.

Node 0 is modeled as a slack bus. The other nodes are modeled as PQ buses for which the injected complex power are specified. The admittance matrix can be partitioned as
\begin{equation}\nonumber
\begin{bmatrix}
I_0^t\\
\mathbf{i}^t
\end{bmatrix} =
\begin{bmatrix}
y_{00} & \bar{y}^\intercal  \\
\bar{y} & \mathbf{Y}
\end{bmatrix}
\begin{bmatrix}
V_0\\
\mathbf{v}^t
\end{bmatrix}.
\end{equation}
The net complex power injection is then
\begin{equation}
\mathbf{s} = \textrm{diag}(\mathbf{v})\Big(\mathbf{Y}^*(\mathbf{v})^* + \bar{y}^*(\mathbf{v}_0)^* \Big).\nonumber
\end{equation}

We define a vector $\mathbf{r} \in \mathbf{R}^M$ for certain (combined) electrical quantities of interests (e.g., voltage magnitudes, current injections, power injection at the substation, etc.) as a function of nodal power injections $\mathbf{p},\mathbf{q}$:
\begin{equation}
\label{powerflow}
    \mathbf{r} = f(\mathbf{p},\mathbf{q})
\end{equation}
with $f(\cdot)$ representing the nonlinear power flow and all node injections collected in a vector compact form of $\mathbf{p} :=[p_1,\dots,p_N]^\intercal$ and $\mathbf{q} :=[q_1,\dots,q_N]^\intercal$. 

Due to the convexity and computation concerns, here we leverage the linearization of \eqref{powerflow} as follows: 
\begin{equation}
\label{linear_powerflow}
    \mathbf{r} = \mathbf{A}\mathbf{p} + \mathbf{B}\mathbf{q} + \mathbf{r}_0,
\end{equation}
where the parameters $\mathbf{A}$, $\mathbf{B}$ and $\mathbf{r}_0$ can be attained from various linearization methods, e.g., \cite{guggilam2016scalable, bernstein2017linear,gan2016online}.

For simplicity, the network is considered as symmetric in steady state, where all currents and voltages are sinusoidal signals at the same frequency. Nevertheless, the presented model and results can be readily extended to unbalanced multi-phase systems. 

\section{Solving OPF with Measurement Feedback}\label{sec:opf}
In this section, we introduce a general OPF problem and the pertinent gradient algorithm with measurement feedback from nonlinear power flow to reduce modeling errors. 

\subsection{OPF Formulation}
Consider the following OPF problem for distribution networks
\begin{subequations}\label{eq:opt}
\begin{eqnarray}
\textbf{P1-OPF:} & \underset{\mathbf{p},\mathbf{q}}{\min} & \sum_{i\in\mathcal{N}}C_i(p_i,q_i)+ C_0(\mathbf{p},\mathbf{q}),\\
&\text{s.t.}&\mathbf{g}(\mathbf{r}(\mathbf{p},\mathbf{q})) \leq 0, \label{eq:voltreg}\\
&& (p_i,q_i)\in\mathcal{Z}_i,\forall i\in\mathcal{N},\label{eq:X} 
\end{eqnarray}
\end{subequations}
where $\mathbf{r}(\mathbf{p},\mathbf{q})$ represents the linearized relationship~\eqref{linear_powerflow}, $C_0(\mathbf{p},\mathbf{q})$ is cost function capturing system-wise objectives (e.g., cost of deviation of total power injections into the substation from nominal values), associated with the local objective function $C_i(p_i,q_i)$ that captures the generation costs, ramping costs, active power losses, renewable curtailment penalty, auxiliary service expenses and reactive compensation (comprising a weighted sum thereof) at node $i \in \mathcal{N}$. The general inequality $\mathbf{g}(\cdot)$ as functions of nodal injections $(\mathbf{p},\mathbf{q})$ and interested electrical quantity $\mathbf{r}$ describe overall network coupling and inequality constraints, such as power flow, voltage magnitude bounds. We have the following assumption
\begin{assumption}\label{assumption_1}
A set of local objective functions $C_i(p_i,q_i), \forall i\in \mathcal{N}$ are continuously differentiable and strongly convex functions of $(p_i,q_i)$, and their first-order derivatives are bounded within their operation regions indicated as $(p_i,q_i) \in \mathcal{Z}_i, \forall i\in \mathcal{N}$. The system-wise objective function $C_0(\mathbf{p},\mathbf{q})$ is continuously differentiable and convex with its first-order derivative bounded. The function $\mathbf{g}(\cdot)$ is convex with bounded derivatives.
\end{assumption}

We constrain our electrical quantity vector $\mathbf{r}$ through a prescribed continuously differentiable and convex function $\mathbf{g}(\cdot)$ with bounded derivatives. The node injections are subject to the convex and compact sets defined as $\mathcal{Z}_i:=\{ (p_i,q_i) | \underline{p}_i\leq p_i \leq \overline{p}_i,\ \underline{q}_i\leq q_i \leq \overline{q}_i\}$. Note that the feasible regions might depend on the inherent terminal properties of various dispatchable devices, e.g., inverter-based distributed generators, energy storage systems or small-scale diesel generators, such that the active and reactive power $(p_i, q_i)$ are additionally subjected to the apparent power limitation or active power availability. For the simplification and generalization, we put a box constraint on active and reactive power. In general, this constraint can be specified to each individual device's operation region. We assume strict feasibility of problem \eqref{eq:opt}.

\begin{assumption}[Slater's Condition] There exist a strictly feasible point within the operation region $(\bar{\mathbf{p}},\bar{\mathbf{q}}) \in \mathcal{Z}$, where $\mathcal{Z}:=\mathcal{Z}_1 \times,\dots,\times \mathcal{Z}_N$, so that
\begin{equation}
    \mathbf{g}(\mathbf{r}(\mathbf{\bar{p}},\mathbf{\bar{q}})) < 0.\nonumber
\end{equation}
\end{assumption}

\subsection{Primal-Dual Gradient Algorithm}
The regularized Lagrangian
$\mathcal{L}$ for \eqref{eq:opt} is
\begin{equation}\label{eq:L_opf}
\mathcal{L}  =  \sum_{i\in\mathcal{N}}C_i(p_i,q_i) + C_0(\mathbf{p},\mathbf{q}) + \bm{\mu}^\intercal \mathbf{g}(\mathbf{r}(\mathbf{p},\mathbf{q})) - \frac{\eta}{2}\|\bm{\mu}\|_2^2,
\end{equation}
where $\bm{\mu}$ is the dual variable vector for the inequality constraints. To promote a provable convergence property, the Lagrangian \eqref{eq:L_opf} includes a Tikhonov regularization term $-\frac{\eta}{2}\|\bm{\mu}\|_2^2$ with a prescribed small parameter $\eta$ that introduces bounded discrepancy \cite{simonetto2014double,koshal2011multiuser}. Then we come to the following saddle-point problem
\begin{equation}\label{eq:maxmin_L}
    \max_{\bm{\mu}\in\mathbf{R}_{+}} \min_{(\mathbf{p},\mathbf{q})\in\mathcal{Z}} \mathcal{L} \left(\mathbf{p}, \mathbf{q},\bm{\mu} \right),
\end{equation}
and an iterative primal-dual gradient algorithm to reach the unique saddle-point of \eqref{eq:maxmin_L}
\begin{subequations}\label{eq:primaldual}
\begin{eqnarray}
& \mathbf{p}^{k+1}&=\left[\mathbf{p}^k -\epsilon \nabla_{\mathbf{p}}\mathcal{L}|_{\mathbf{p}^k,\mathbf{q}^k,\bm{\mu}^k}\right]_{\mathcal{Z}},\label{eq:gradient_opf_p}\\
& \mathbf{q}^{k+1}&=\left[\mathbf{q}^k -\epsilon \nabla_{\mathbf{q}}\mathcal{L}|_{\mathbf{p}^k,\mathbf{q}^k,\bm{\mu}^k}\right]_{\mathcal{Z}},\label{eq:gradient_opf_q}\\
& \bm{\mu}^{k+1}&=\left[\bm{\mu}^k + \epsilon \nabla_{\bm{\mu}}\mathcal{L}|_{\mathbf{r}^k}\right]_{\mathbf{R}_+},\label{eq:gradient_opf_dual}\\
& \mathbf{r}^{k+1}&=\mathbf{A}\mathbf{p}^{k+1} + \mathbf{B}\mathbf{q}^{k+1} + \mathbf{r}_0,\label{eq:gradient_opf_pf}
\end{eqnarray}
\end{subequations}
where $\epsilon \in \mathbf{R}_{++}$ is a constant stepsize to be determined, and the operators $[\cdot]_{\mathcal{Z}}$ and $[\cdot]_{\mathbf{R}_+}$ project objects onto the feasible set $\mathcal{Z}:= \times_{i\in\mathcal{N}}\mathcal{Z}_i$ and non-negative orthant, respectively. For notation convenience, we compact the gradient operations in \eqref{eq:primaldual} as an operator $\mathcal{T}$ given by
\begin{equation}\label{eq:linear_operator}
   \mathcal{T}\left(
    \begin{bmatrix}
    \mathbf{p}^{k}\\
    \mathbf{q}^{k}\\
    \bm{\mu}^{k}
    \end{bmatrix}
    \right) :=
    \begin{bmatrix}
    \epsilon \nabla_{\mathbf{p}}\mathcal{L}|_{\mathbf{p}^k,\mathbf{q}^k,\bm{\mu}^k} \\
    \epsilon \nabla_{\mathbf{q}}\mathcal{L}|_{\mathbf{p}^k,\mathbf{q}^k,\bm{\mu}^k}\\
    - \epsilon \nabla_{\mathbf{\bm{\mu}}}\mathcal{L}|_{\mathbf{p}^k,\mathbf{q}^k,\bm{\mu}^k}
    \end{bmatrix},
\end{equation}
to recast \eqref{eq:primaldual} as 
\begin{equation}
    \begin{bmatrix}
    \mathbf{p}^{k+1}\\
    \mathbf{q}^{k+1}\\
    \bm{\mu}^{k+1}
    \end{bmatrix}  =  \left[\begin{bmatrix}
    \mathbf{p}^{k}\\
    \mathbf{q}^{k}\\
    \bm{\mu}^{k}
    \end{bmatrix} - \mathcal{T}\left(
    \begin{bmatrix}
    \mathbf{p}^{k}\\
    \mathbf{q}^{k}\\
    \bm{\mu}^{k}
    \end{bmatrix}
    \right)\right]_{\mathcal{Z}\times \mathbf{R}_{+}}.
\end{equation}
We define a compact vector $\mathbf{x}:=[(\mathbf{p}^k)^\intercal,(\mathbf{q}^k)^\intercal,({\bm{\mu}}^k)^\intercal]^\intercal$ for notation convenience. Under Assumption \ref{assumption_1}, for any feasible points $\mathbf{x}_1$ and $\mathbf{x}_2$, we can show that $\mathcal{T}$ is $L$-Lipschitz continuous and $M$-strongly monotone with constants $L, M >0$, and satisfies the following inequalities
\begin{equation}\label{eq:property_T_1}
    \|\mathcal{T}(\mathbf{x}_1) - \mathcal{T}(\mathbf{x}_2) \|_2^2 \leq L^2 \|\mathbf{x}_1 - \mathbf{x}_2\|_2^2,
\end{equation}
\begin{equation}\label{eq:property_T_2}
\left(\mathcal{T}(\mathbf{x}_1) - \mathcal{T}(\mathbf{x}_2)\right)^\intercal \left(\mathbf{x}_1 -
\mathbf{x}_2\right) \geq M\|\mathbf{x}_1 - \mathbf{x}_2\|_2^2.
\end{equation}
We can obtain the following convergence condition for \eqref{eq:primaldual}.
\begin{theorem} \label{theorem_1}
If the stepsize $\epsilon$ satisfies
\begin{equation}
0 < \epsilon < 2M/L^2,
\end{equation}
algorithm \eqref{eq:primaldual} converges to the unique saddle point of \eqref{eq:L_opf}. 
\end{theorem}

We refer its proof to the Appendix. 

\subsection{Feedback from Nonlinear Power Flow}
The optimization problem \eqref{eq:opt} and its gradient algorithm \eqref{eq:primaldual} are based on the linearized power flow to guarantee its convexity, design computationally tractable algorithms, and prove the convergence of the algorithm to the saddle point. However, the linearization errors render the optimum different from the ones with nonlinear power flow, or the set points infeasible with nonlinear power flow. To address this issue, one paradigm is used to leverage feedback-based online optimization methods \cite{colombino2019online,dall2016optimal,zhou2017incentive} to reduce the modeling error. Particularly, we replace \eqref{eq:gradient_opf_pf} with the following nonlinear power flow
\begin{eqnarray}
\mathbf{r}^{k+1} &=& f(\mathbf{p}^k,\mathbf{q}^k)
\end{eqnarray}
that is executed by the physical system, and use the measured values $\mathbf{r}^{k+1}$ from the physical system to update dual variables \eqref{eq:gradient_opf_dual} in the next iteration. Convergence to a bounded range of the optimum can be analytically shown for such implementation. This also facilitates real-time implementation that can track the time-varying optimal set points \cite{dall2016optimal,zhou2017incentive}.

However, one crucial issue of such feedback-based algorithm has been largely overlooked: in practice, there are too few monitoring devices in distribution systems to measure all components of $\mathbf{r}$, and therefore it is impossible to implement feedback based algorithms to solve the problem \eqref{eq:opt}. 
Take the optimal voltage regulation problem (formulated in Appendix) used in Section~\ref{sec:num} as an example. Given voltage measurement deployed at three nodes only, dual variables can only be updated at these three locations, leaving the most of the 37-node system unobserved. As a result, the obtained solution may not be optimal or feasible for the original optimization problem. For instance, voltage violation may take place at unobserved locations; see Fig.~\ref{fig:results_comparison}. This motivates us to integrate state estimation into OPF solution methods for a better observed system.


\section{Solving OPF with State Estimation Feedback}\label{sec:se}
To approach the actual operation scheme of distribution networks and also improve the observability and performance of existing OPF controllers without loss of computational effects, we additionally perform a state estimation based on available measurements, solving the sub-problem \eqref{eq:se_general}, before we update the dual variables \eqref{eq:gradient_opf_dual}. In principle, this allows us to feedback the ``full picture" of network real-time responses with the latest information and in turn compute the optimal decisions without any ``blind spots" of the grid. We use the following control diagram in Fig.\ref{fig:OPF_SE_diagram_1} to illustrate the mechanism of the proposed OPF framework with a state estimation feedback loop. 

\begin{figure}[!htbp]
\centering
\includegraphics[width=3.5in]{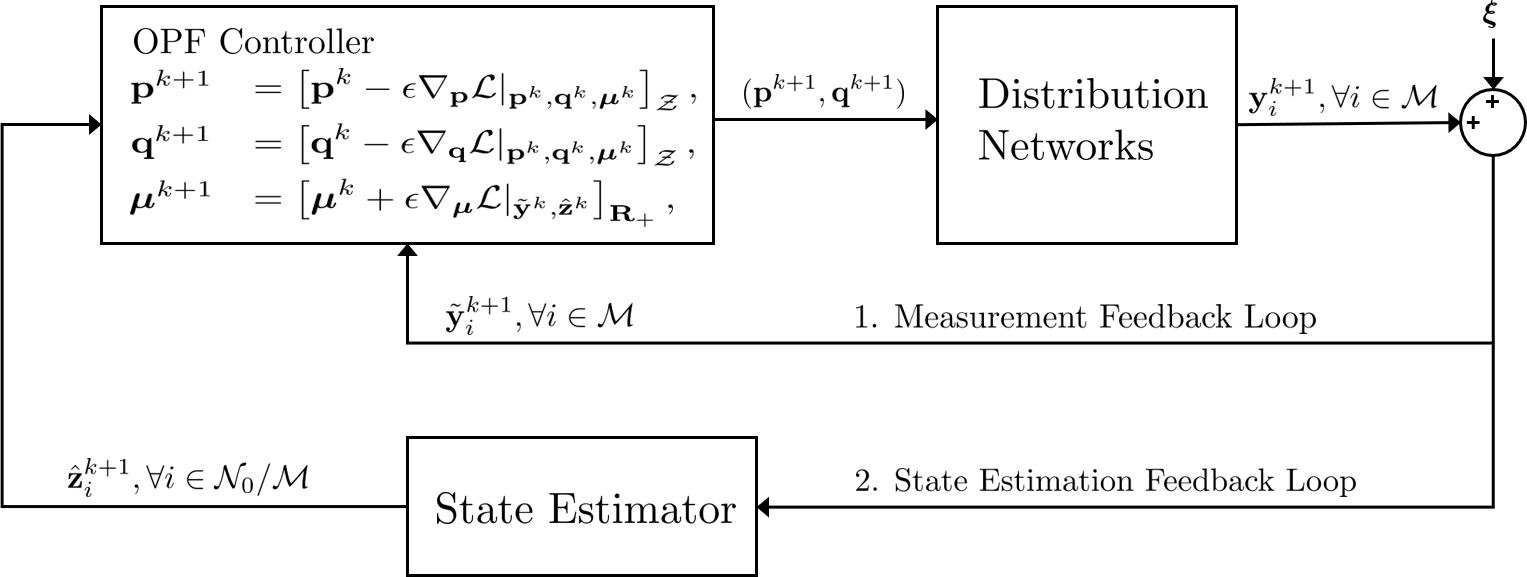}
\caption{The diagram of the proposed optimal power flow problem with state estimation feedback. In this case, we break the timescale to loop an off-line state estimation phase into a gradient-based OPF controller, due to the extreme computational efficiency of state estimation problems. Two feedback loops feed-in the online measurement $\tilde{\mathbf{y}}^{k+1}$ and estimation results $\hat{\mathbf{z}}^{k+1}$ to the OPF controller for the next gradient direction. We also assume that some of the measurement noises are less than the estimation errors to keep the online measurement loop.
}
\label{fig:OPF_SE_diagram_1}
\end{figure}

\subsection{State Estimation Problem}
We consider the following grid measurement model
\begin{equation*}
    \mathbf{y} = h(\mathbf{z}) + \bm{\xi},
\end{equation*}
where $\mathbf{z}$ is the true system states and $\mathbf{y}$ is the measurement vector, and $\bm{\xi}$ is measurement error, assumed to be distributed according to a normal distribution $\mathbb{N}(0, \Sigma)$. The measurement vector  $\mathbf{y}$ comprises both a limited set of sensor measurements and a set of so-called ``pseudo-measurement'', which subjected to Gaussian distributions based on the historical data (e.g., customer billing data and typical load profile) to roughly estimate power injections \cite{gomez2004power}.

To estimate the actual grid states from the available measurements, we consider a weighted least squares (WLS) estimator \cite{gomez2004power,wu1990power,zhou2019graident}, as follow
\begin{equation}\label{eq:se_general}
\textbf{P2-SE:} ~~\underset{\mathbf{z}}{\min}  \frac{1}{2}\left(\mathbf{y} -h(\mathbf{z})\right)^\intercal W\left(\mathbf{y} -h(\mathbf{z})\right),
\end{equation}
 where the weight matrix is defined as $W = \Sigma^{-1}$. When the measurement model is nonlinear, this may be a difficult optimization problem to solve exactly; to ensure convergence and appropriate estimation algorithm should be carefully designed. In this case, we will subsequently introduce an assumption on the solution to this problem. With a linear measurement function, a solution to the estimation problem can be guaranteed provided sufficiently many sensor measurements, leading to grid observability (defined below).

To emphasize, this SE feedback problem must be appropriately specified to fit the various optimal control purposes of OPF problems. In particular for the OPF problem \eqref{eq:opt}, the true system states boil down to active and reactive power injections, denoted in a compact form as $\mathbf{z}:= [\mathbf{p}^\intercal,\mathbf{q}^\intercal]^\intercal\in\mathbf{R}^{2N}$, and all electrical quantities of interest $\mathbf{r}$ are uniquely determined by power flow coupling. The allocation of measurement $\mathcal{M}$ for distribution networks associating with measurement quantities $\mathbf{y}$ should meet the full observation requirement, which will guarantee the unique system states $(\mathbf{p},\mathbf{q})$. We will specify the formulation \eqref{eq:opt} and \eqref{eq:se_general} into a voltage regulation problem in Section~\ref{sec:num}, associated with its detailed formulation in the Appendix.
\begin{definition}[Full Observability]
The system is fully observable\footnote{This definition should be distinguished from observability of linear dynamical systems. Here, we limit the definition of observability to static SE problems throughout this manuscript\cite{wu1985network}.} (100\%) if $\mathbf{z}=0$ is the only solution for $h(\mathbf{z}) = 0$. Otherwise, the system has unobservable states, which leads to multiple solutions.
\end{definition}

\begin{assumption}
The WLS state estimation problem \eqref{eq:se_general} for distribution networks is fully observable with appropriate measurement allocation $\mathcal{M}$ and associated measurements $\mathbf{y}$.  
\end{assumption}



\subsection{State Estimation as Feedback}

Fig.~\ref{fig:OPF_SE_diagram_1} illustrates the proposed OPF controller with the state-estimation feedback loop. We assume that the measurement noises are less than the state estimation errors, and have small variances\footnote{This is an interesting assumption when we proceed our numerical tests. If some of the system states can be reached from both SE and noisy measurement, which one we should trust more, measurement or estimation? In fact, estimation errors explicitly depend on multiple factors, e.g., choose of estimators, sensor allocation $\mathcal{M}$ associating with measurement selection $\mathbf{y}$. We leave an open question here to inspire more discussions in future from both analytical and numerical perspectives.},  and then at the nodes with monitoring $\forall i \in\mathcal{M}$,  we utilize the measurement information to OPF controller instead of exploiting estimation results. This proposed methodology can be formally described in Algorithm \ref{algorithm:OPF-SE algorithm}.

\begin{algorithm}
    \caption{(OPF with SE Feedback Algorithm)}\label{algorithm:OPF-SE algorithm}
    \begin{algorithmic}[1]
        \Require nominal pattern of Netloads $(\mathbf{p}^0,\mathbf{q}^0) \in
        \mathcal{N}$
        \For{$k = 0:K$}
            \State $\mathbf{p}^{k+1} = \left[\mathbf{p}^k - \epsilon \nabla_{\mathbf{p}} \mathcal{L}(\mathbf{p}^k,\mathbf{q}^k,
            \bm{\mu}^k) \right]_{\mathcal{Z}}$
            \State $\mathbf{q}^{k+1} = \left[\mathbf{q}^k - \epsilon \nabla_{\mathbf{q}} \mathcal{L}(\mathbf{p}^k,\mathbf{q}^k,
            \bm{\mu}^k) \right]_{\mathcal{Z}}$
            \State $\mathbf{r}^{k+1} \leftarrow  \textrm{nonlinear power flow}  ~(\mathbf{p}^{k+1}, \mathbf{q}^{k+1})$
            \State system measurement $\tilde{\mathbf{y}}_i^{k+1}, \forall i \in \mathcal{M}$
            \State state estimation of system response
            $\hat{\mathbf{z}}_i^{k+1}, \forall i \in \mathcal{N}_0
            /\mathcal{M}$
            \State $\bm{\mu}^{k+1} =  \left[\bm{\mu}^{k} + \epsilon\nabla_{\bm{\mu}} \mathcal{L}(\mathbf{\tilde{y}}^{k+1},\hat{\mathbf{z}}^{k+1})\right]_{\mathbf{R}_+}$
        \EndFor
\end{algorithmic}
\end{algorithm} 
\subsection{Convergence Analysis}
With the state estimator (state estimation) feedback loop for the OPF controller shown in Fig.~\ref{fig:OPF_SE_diagram_1}, we propose a gradient-based OPF controller with SE feedback, which is described in Algorithm 1. At each time-slot, the system pursuit by the following steps
\begin{subequations}\label{eq:gradient_SE_OPF_Nonlinear}
\begin{equation}\label{eq:gradient_p}
\begin{aligned}
& \mathbf{p}^{k+1}  \\
& = \left[\mathbf{p}^{k} - \epsilon \left(\nabla_{\mathbf{p}}C(\mathbf{p}^{k},\mathbf{q}^{k}) + \nabla_{\mathbf{p}}C_0(\mathbf{p}^{k},\mathbf{q}^{k})+\mathbf{A}^\intercal\bm{\mu}^k 
    \right)\right]_{\mathcal{Z}},
\end{aligned}
\end{equation}
\begin{equation}\label{eq:gradient_q}
\begin{aligned}
& \mathbf{q}^{k+1} \\
& =  \left[\mathbf{q}^{k} - \epsilon \left(
    \nabla_{\mathbf{q}}C(\mathbf{p}^{k},\mathbf{q}^{k}) + \nabla_{\mathbf{q}}C_0(\mathbf{p}^{k},\mathbf{q}^{k})+\mathbf{B}^\intercal \bm{\mu}^k
    \right)\right]_{\mathcal{Z}},
\end{aligned}
\end{equation}
\begin{equation}\label{eq:gradient_mu_upper}
\begin{aligned}
\bm{\mu}^{k+1} = \left[\bm{\mu}^k + \epsilon \left(\mathbf{g}(\mathbf{r}^k) - \eta\bm{\mu}^k\right) \right]_{\mathbf{R}_{+}},~~~~~~~~~~~~~~~~~~~~~~~
\end{aligned}
\end{equation}
\begin{equation}\label{eq:gradient_voltage_from_SE}
\begin{aligned}
    & \mathbf{r}^{k+1} \leftarrow (\hat{\mathbf{z}}^{k+1},\tilde{\mathbf{y}}^{k+1}),\\ & \textrm{(jointly updated from state estimator and measurement)}.~
\end{aligned}
\end{equation}
\end{subequations}
The above iteration \eqref{eq:gradient_SE_OPF_Nonlinear} is performed to its convergence. At $k$-th iteration, we jointly update the system states based on the state estimation results in \eqref{eq:gradient_voltage_from_SE} and the measurement $\tilde{\mathbf{y}}$. These estimation results are attained from the optimizer of \eqref{eq:se_general}, which enable the feasibility of the nonlinear power flow constraints \eqref{powerflow}.

\begin{assumption}
At $k$-th iteration of OPF controller during any time slot, there exists a constant $0< \psi < +\infty$ to bound the difference between the actual system states and the states attained from various state estimators \eqref{eq:se_general}
such that, $\| \mathbf{z}^{k} - \mathbf{\hat{z}}^{k*} \|_2 < \psi $.
\end{assumption}

As the state estimation in distribution networks has been widely discussed for different applications \cite{primadianto2016review}, the existing literature shows that these type of methods lead to an accurate and computational-efficient approximation under nominal operating condition. This allows us to expect a small bound constant $\psi$. 

\begin{lemma}\label{lemma:nonlinear difference}
We here define a nonlinear operator $\mathcal{T}_{\textrm{NL}}$ to represent the iteration \eqref{eq:gradient_SE_OPF_Nonlinear}, which can be regarded as a counterpart operator of the linear operator $\mathcal{T}$ in \eqref{eq:linear_operator}. Then, there exist some constant $\rho > 0$ for the following condition hold
\begin{equation}\label{eq:Linear_NonLinear_difference}
    \|\mathcal{T}(\mathbf{x}^k) - \mathcal{T}_{\textrm{NL}}(\mathbf{x}^k)\|_2^2 \leq \rho,
\end{equation}
where $\{\mathbf{x}^k | ~ \mathbf{x}^k := [(\mathbf{p}^k)^\intercal, (\mathbf{q}^k)^\intercal, (\bm{\mu}^k)^\intercal ]^\intercal, k \leq K \}$ is the sequence generated by Algorithm 1.
\end{lemma}

We omit the proof here and refer the interested readers to \cite{bertsekas1989parallel,zhou2017incentive}. It is worth noting that this condition holds only under Assumption 4, so that estimation error bound $\psi$ informs the constant $\rho$. 


\begin{theorem}\label{theorem_2}
Given the prescribed step size $\epsilon$ based on Theorem 1, and under Assumptions 1-4, the sequence $\{\mathbf{x}^k | ~ \mathbf{x}^k := [(\mathbf{p}^k)^\intercal, (\mathbf{q}^k)^\intercal, (\bm{\mu}^k)^\intercal ]^\intercal, k \leq K \}$ generated by Algorithm \ref{algorithm:OPF-SE algorithm} is bounded by
\begin{equation}\label{eq:upper_bounds_Nonlinear}
    \lim_{K\to \infty} \sup ~\|\mathbf{x}^K - \mathbf{x}^* \|_2^2 = \frac{\epsilon^2 \rho}{2\epsilon M - \epsilon^2L^2 },
\end{equation}
where $\mathbf{x}^*=[(\mathbf{p}^*)^\intercal, (\mathbf{q}^*)^\intercal, (\bm{\mu}^*)^\intercal ]^\intercal$ is the optimizer of $\mathcal{L}$ in \eqref{eq:maxmin_L}.
\end{theorem}

\begin{proof}
Given the operator $\mathcal{T}$ Lipschitz continuous and strongly monotone \eqref{eq:property_T_1}-\eqref{eq:property_T_2} and under the Assumption 1-4 and Lemma \ref{lemma:nonlinear difference},
we show the discrepancy between the optimizer $\mathbf{x}^*$ of \eqref{eq:maxmin_L} and the sequence $\{\mathbf{x}^k\}$ at $K$-th iteration
\begin{align}
      & \|\mathbf{x}^K - \mathbf{x}^*\|_2^2 \nonumber \\
    \leq ~& \|\mathbf{x}^{K-1} -\epsilon\mathcal{T}_{\textrm{NL}}(\mathbf{x}^{K-1}) - \mathbf{x}^* + \epsilon \mathcal{T}(\mathbf{x}^*) \|_2^2 \nonumber\\
    = ~& \|\mathbf{x}^{K-1} -\epsilon\mathcal{T}_{\textrm{NL}}(\mathbf{x}^{K-1}) - \epsilon\mathcal{T}(\mathbf{x}^{K-1}) \nonumber \\
    & ~~~~~~~~~~~~~~~~~~~~~~~ + \epsilon\mathcal{T}(\mathbf{x}^{K-1}) - \mathbf{x}^* + \epsilon \mathcal{T}(\mathbf{x}^*) \|_2^2 \nonumber\\
    \leq ~& \|\mathbf{x}^{K-1} -\epsilon \mathcal{T}(\mathbf{x}^{K-1}) - \mathbf{x}^* + \epsilon\mathcal{T}(\mathbf{x}^*) \|_2^2 + \epsilon^2 \rho\nonumber\\
    = ~& \|\mathbf{x}^{K-1} - \mathbf{x}^* \|_2^2 + \|\epsilon \mathcal{T}(\mathbf{x}^{K-1}) - \epsilon \mathcal{T}(\mathbf{x}^*)\|_2^2\nonumber\\
    & ~~~~~~~ -2\epsilon (\mathcal{T}(\mathbf{x}^{K-1}) - \mathcal{T}(\mathbf{x}^*))^\intercal (\mathbf{x}^{K-1} - \mathbf{x}^*\nonumber) + \epsilon^2\rho\\
    \leq ~& \left(\epsilon^2 L^2 - 2\epsilon M + 1 \right) \|\mathbf{x}^{k-1} - \mathbf{x}^*\|_2^2 + \epsilon^2 \rho. \label{eq:recursively step}
\end{align}
The non-expensiveness of projection operator results in the first inequality. The second inequality comes from Lemma \ref{lemma:nonlinear difference} and the last inequality is based on \eqref{eq:property_T_1}-\eqref{eq:property_T_2}. Let $\Gamma = \epsilon^2L^2 - 2\epsilon M + 1$ and recursively implement this relationship in \eqref{eq:recursively step} backwards to the initial step, then comes to
\begin{align}
    & \|\mathbf{x}^K - \mathbf{x}^*\|_2^2 \leq \Gamma^K \|\mathbf{x}^{0} - \mathbf{x}^*\|_2^2 + \epsilon^2 \rho \left( \frac{1-\Gamma^K}{1-\Gamma} \right). \label{eq:recursively initial step}
\end{align}
We have any step size chosen as $\epsilon < \frac{2M}{L^2}$, in Theorem \ref{theorem_1}, which leads $0 < \Gamma < 1 $. For such $\Gamma$ and any initial condition $x^0 \in \mathbf{R}$, when $K \to \infty$, we come to an upper bound of this discrepancy
\begin{equation}
    \lim_{K\to \infty}\sup ~ \|\mathbf{x}^K - \mathbf{x}^* \|_2^2 = \frac{\epsilon^2 \rho}{2\epsilon M - \epsilon^2L^2 },
\end{equation}
which concludes the proof.
\end{proof}

The condition \eqref{eq:upper_bounds_Nonlinear} from Theorem \ref{theorem_2} provides an upper bound on the distance between the sequence $\{\mathbf{x}^k | ~ \mathbf{x}^k := [(\mathbf{p}^k)^\intercal, (\mathbf{q}^k)^\intercal, (\bm{\mu}^k)^\intercal ]^\intercal, k \leq K, K \to \infty \}$ generated by our proposed OPF with SE feedback algorithm \eqref{eq:gradient_SE_OPF_Nonlinear} and the saddle point $\mathbf{x}^*$ of \eqref{eq:maxmin_L}. This analytical bound indicates that our proposed control diagram in Fig.~\ref{fig:OPF_SE_diagram_1} has robust performance to estimation errors and measurement noises
\begin{itemize}
    \item \emph{Inherent measurement noises:} The online measurements by PMUs are typically within $1\% \sim 2\%$ of actual values. The pseudo measurements of active and reactive power can be regarded as a rough initialization (i.e., up-to 50\% variations in comparison to actual values). These errors can be reduced through the estimation phase in \eqref{eq:se_general}, which improves decisions from the OPF controller with SE feedback \eqref{eq:gradient_voltage_from_SE} , yielding better robustness to measurement noises and power variability; 
    
    \item \emph{Linearization approximation errors:} The OPF-phase (and some of state estimators) in the proposed algorithm utilize the linear power flow to promote the affordable-computational efficiency. The discrepancy between linear operator $\mathcal{T}$ and nonlinear operator $\mathcal{T}_{\textrm{NL}}$ has been quantified in \eqref{eq:Linear_NonLinear_difference} by $\rho$, which incorporates nonlinear power flow of distribution networks when implementing the set-points $(\mathbf{p}^k, \mathbf{q}^k)$ for controllable devices.
\end{itemize}
In the next section, we demonstrate the effectiveness and flexibility of the proposed OPF framework with SE feedback for its efficient convergence and robustness performance on the IEEE 37-node distribution network.

\section{Numerical Results}\label{sec:num}

In this section, we apply our OPF framework with SE feedback algorithm to mitigate an overvoltage situation and minimize the PV curtailment in distribution networks by controlling set points of renewable energy resources (RESs) through limited measurement of voltage magnitudes. We admit the electrical quantities of interests down to voltage magnitude, and formulate an OPF problem for voltage regulation. In the state estimation phase, we use pseudo measurement of nodal power injections incorporating with limited number of voltage magnitude measurement to estimate the whole voltage magnitude of distribution networks. The detailed mathematical formulations are given in Appendix. At each time step, the set points of controllable RESs are repeatedly optimized along the gradient direction until convergence under the feedback of the fully observable voltage updates from voltage estimator.

We use a modified IEEE 37-node test feeder \cite{zimmerman2010matpower} shown in Fig.~\ref{fig:37-node feeder} to demonstrate our proposed OPF with SE feedback algorithm. The network is three-phase balanced and we place 21 photovoltaic (PV) distributed generators, whose locations are marked by boxes in Fig.~\ref{fig:37-node feeder}. The stars also indicate three actual voltage magnitude measurement, whose sensor noises follow the independent normal distributions with zero mean and 1\% standard deviation of their actual values. All nodes have pseudo-measurement of power injections, whose errors are subjected to the independent normal distribution with zero mean and 50\% standard deviation of their nominal values based on historical data. In general, this framework can be adapted to the time-varying systems with heterogeneous inputs. To simplify presentation, we show the results in a single time step (e.g., 12:00 PM), which fixes the inputs and runs the algorithm to convergence. The real-time measurement of solar irradiation and load profile is derived from the measurement of the feeders in Anatolia, California \cite{Bank13}. The simulation is implemented on a laptop with Intel Core i7-6600U CPU@2.6GHz 8.00GB RAM, using MATLAB and MATPOWER \cite{zimmerman2010matpower}.

\begin{figure}
    \centering
    \includegraphics[width=2in]{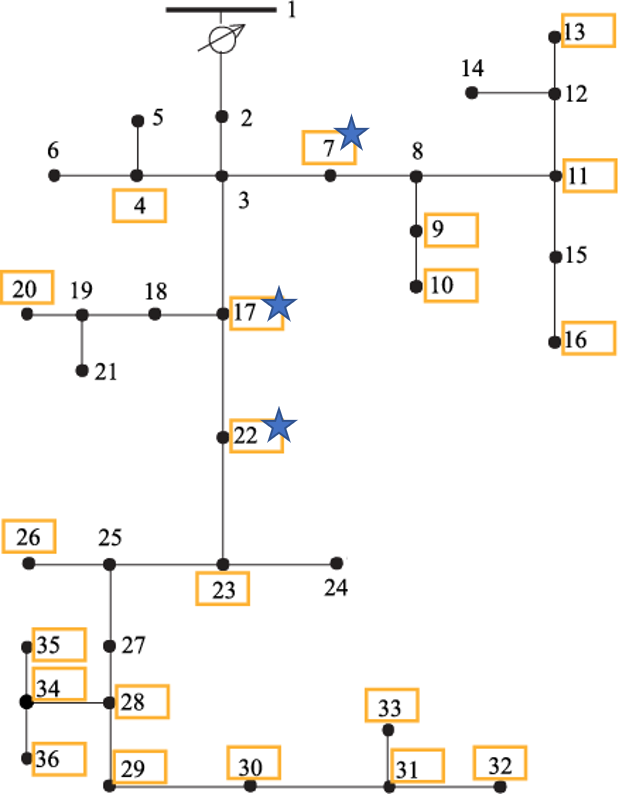}
    \caption{IEEE 37-node test feeder. The stars indicate the voltage sensor placement.}
    \label{fig:37-node feeder}
\end{figure}

\begin{figure}
    \centering
    \includegraphics[width=2.85in]{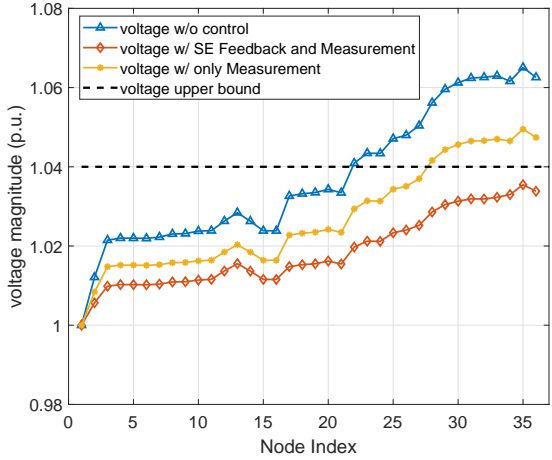}
    \caption{Voltage magnitude profiles comparison under different control schemes. There is a small gap between the voltage of node \#35 and voltage upper bound due to less then 1\% estimation errors in average.}
    \label{fig:results_comparison}
\end{figure}

For comparison purposes, we also introduce another scenario to show the limits of the OPF controller only having the knowledge of the voltage measurement instead of the whole system-wise voltage estimation. Fig.~\ref{fig:results_comparison} shows the solutions of the proposed OPF controller with SE feedback. It can be seen that the proposed OPF controller successfully mitigates the overvoltage situation under the ``full" picture of voltage information. The results illustrate the proposed OPF controller's performance guarantee with a limited number of sensors.

Fig.~\ref{fig:results_comparison} also demonstrates the performance degradation of the OPF controller with only voltage measurements (no state estimation feedback). Without a whole description of voltage magnitude, the OPF controller has limited knowledge of overvoltage at certain nodes. Only the voltage at nodes having voltage monitoring can be accurately optimized, and the large voltage uncertainties at other nodes can threaten safe operation. To emphasize, it is unrealistic to have a complete voltage measurements, and it is prohibitive to install sensors at all end-users locations due to privacy and cost concerns. Under this circumstance, there are always ``blind" points in distribution networks facing overvoltage/undervoltage situation. Unfortunately, our hands are tied to directly reach out every corner of networks, especially for extremely large networks. This comparison motivates the incorporation of state estimation feedback for the OPF controller as an effective and efficient method to integrate voltage estimates, which guarantees system performance for economic and safe operation.

Fig.~\ref{fig:convergence signal_voltage} and Fig.~\ref{fig:covergence signal_signal} show the convergence of the proposed OPF controller with SE feedback. It is clear that the system has  more robust performance to the noisy measurement and state estimation errors. The SE feedback solver takes 0.2s to reach the optimal voltage estimates, and the OPF controller needs 0.1s to compute its gradient step based on SE results at each iteration. Fig.~\ref{fig:estimation errors} shows the average and maximum errors of voltage estimation at each OPF iteration. We also include the running average for the average and maximum SE errors in Fig.~\ref{fig:estimation errors}. It can be seen that average SE error approaches to 0.67\% and the maximal SE error approaches to 1.37\%, falling in an acceptable range for OPF feedback.

Overall, we conclude that the proposed OPF controller with SE feedback is able to systemically estimate system voltages at unmeasured nodes, successfully mitigating the overvoltage situation and providing robustness to measurement errors and estimation errors. The benefits of closing the loop between OPF control and state estimation can be clearly observed from the perspectives of effectiveness, robustness and efficiency.

\begin{figure}
    \centering
    \includegraphics[width=2.9in]{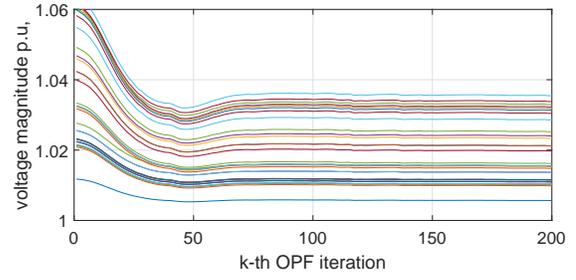}
    \caption{The voltage convergence of the OPF controller with SE Feedback.}
    \label{fig:convergence signal_voltage}
\end{figure}

\begin{figure}
    \centering
    \includegraphics[width=3.1in]{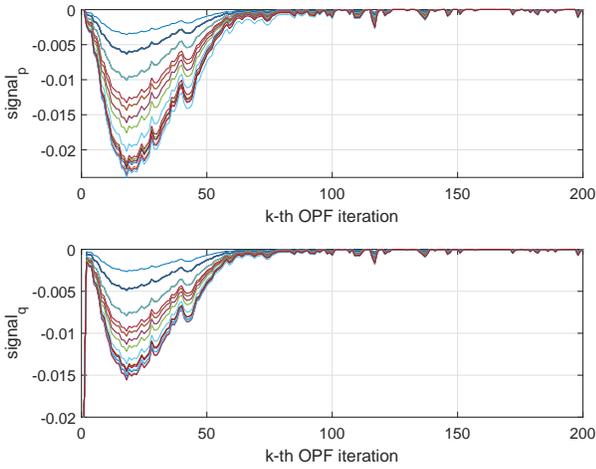}
    \caption{The convergence of voltage violation signals $\mathbf{A}^{\intercal}\bm{\mu}$ and $\mathbf{B}^{\intercal}\bm{\mu}$, which are projected to the gradient updates of inverter setpoints (i.e., active and reactive power).}
    \label{fig:covergence signal_signal}
\end{figure}

\begin{figure}
    \centering
    \includegraphics[width=3.15in]{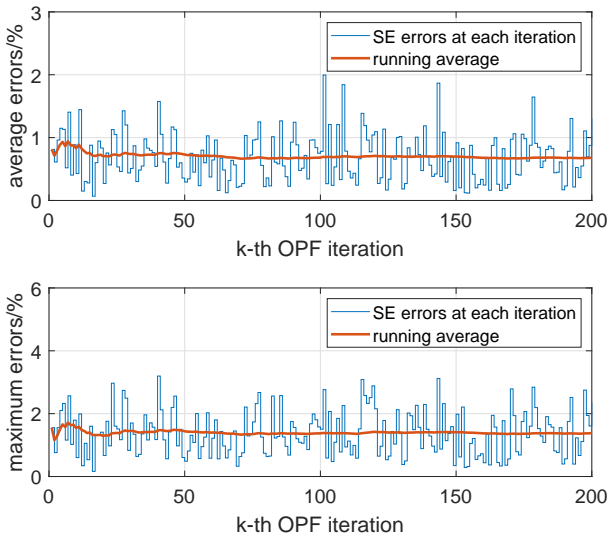}
    \caption{The average and maximum SE errors.}
    \label{fig:estimation errors}
\end{figure}

\section{Conclusions}\label{sec:con}
In this paper, we proposed a general optimal power flow problem with state estimation feedback to facilitate the operation of modern distribution networks. The controller depends explicitly on the state estimation results derived from the system measurement. In contrast to the existing works, our method introduces an additional feedback loop to the OPF controller to estimate the system voltages from a limited number of sensors rather than making strong assumptions on full observability or being ``trapped" by limited system measurements. The performance of our design is analytically characterized and numerically demonstrated. 

Our results on OPF problem provide an initial step towards closing a loop between optimal control and state estimation in power systems, which tractably incorporates state estimation to promote efficient and robust performance.



\appendix

\subsection{Proof of Theorem~\ref{theorem_1}}
\begin{proof}
We have the point of the sequence at $k$-th iteration defined by $\mathbf{x}^{k} :=[(\mathbf{p}^{k})^\intercal, (\mathbf{q}^k)^\intercal, (\bm{\mu}^{k})^\intercal]^\intercal$ and the optimizer of \eqref{eq:maxmin_L} indicated by $\mathbf{x}^{*} :=[(\mathbf{p}^*)^\intercal, (\mathbf{q}^*)^\intercal, (\bm{\mu}^*)^\intercal]^\intercal$. Since the linear operator $\mathcal{T}$ is Lipschitz continuous and strongly monotone \cite{bertsekas1989parallel}, we have following condition held
 and then we show the distance between $\mathbf{x}^{k}$ and the saddle point $\mathbf{x}^*$
\begin{align}
      & \|\mathbf{x}^k - \mathbf{x}^*\|_2^2 \nonumber \\
    \leq ~& \|\mathbf{x}^{k-1} -\epsilon\mathcal{T}(\mathbf{x}^{k-1}) - \mathbf{x}^* + \epsilon \mathcal{T}(\mathbf{x}^*) \|_2^2 \nonumber \\
    = ~& \|\mathbf{x}^{k-1} - \mathbf{x}^* \|_2^2 + \|\epsilon \mathcal{T}(\mathbf{x}^{k-1}) - \epsilon \mathcal{T}(\mathbf{x}^*)\|_2^2 \nonumber \\
    & ~~~~~~~ -2\epsilon (\mathcal{T}(\mathbf{x}^{K-1}) - \mathcal{T}(\mathbf{x}^*))^\intercal (\mathbf{x}^{k-1} - \mathbf{x}^*) \nonumber \\
    \leq ~&\left(\epsilon^2 L^2 - 2\epsilon M + 1 \right) \|\mathbf{x}^{k-1} - \mathbf{x}^* \|,
\end{align}
The non-expensiveness of $\mathcal{T}$ leads to the first inequality and the strongly monotone and Lipschitz continuous properties of $\mathcal{T}$ \eqref{eq:property_T_1}-\eqref{eq:property_T_2} leads to the last inequality. As long as there exist some step size $\epsilon >0$, which let the coefficient $0 < (\epsilon^2 L^2 - 2\epsilon M + 1) < 1$, such that $\epsilon < \frac{2M}{L^2}$, the operator $\mathcal{T}$ converge the sequence $\mathbf{x}^k$ exponentially to the unique saddle point of \eqref{eq:maxmin_L}. This concludes the proof.
\end{proof}

\subsection{Voltage Regulation Problem Formulation}

The general OPF problem \eqref{eq:opt} can be boiled down to a 
voltage magnitude regulation problem, where we specify the electrical quantities vector $\mathbf{r}$ as the voltage magnitude vector denote as $|\mathbf{v}|:=[|v_1|,\ldots,|v_N|]^\intercal$
\begin{subequations}\label{eq:opt_v}
\begin{eqnarray}
\textbf{OPF-V:} & \underset{\mathbf{p},\mathbf{q},|\mathbf{v}|}{\min} & \sum_{i\in\mathcal{N}}C_i(p_i,q_i)+ C_0(\mathbf{p},\mathbf{q}),\\
& \text{s.t.}&|\mathbf{v}|=\mathbf{A}\mathbf{p}+\mathbf{B}\mathbf{q}+ |\mathbf{v}_0|,\\
&&\underline{\mathbf{v}} \leq  |\mathbf{v}| \leq \bar{\mathbf{v}},\\
&& (p_i,q_i)\in\mathcal{Z}_i,\forall i\in\mathcal{N}. 
\end{eqnarray}
\end{subequations}
The inequality constraints capture the lower and upper bounds $(\underline{\mathbf{v}},\bar{\mathbf{v}})$ of voltage magnitudes. In particular, linear approximation for voltage magnitude with AC power flow, as a function of power injection $(\mathbf{p},\mathbf{q})$. The coefficient matrics $(\mathbf{A},\mathbf{B})$ of the linearized voltages and normalized vector $|\mathbf{v}_0|$ can be attained from numerous linearization methods, e.g., \cite{guggilam2016scalable, bernstein2017linear,gan2016online}. The gradient-based OPF controller \eqref{eq:gradient_SE_OPF_Nonlinear} utilizes the online voltage magnitude measurement and voltage estimation to converge the system.

\subsection{Voltage Magnitude Estimation}

We consider the following specific WLS problem for voltage magnitude estimation in distribution networks with the assumption that measurement noises for $p_i$, $q_i$ and $|v_i|$ follows independent Gaussian distributions:
\begin{eqnarray}
& \hspace{-2mm}\textbf{SE-V:} \nonumber  \\ 
& \hspace{-6mm}\underset{\hat{\mathbf{p}},\hat{\mathbf{q}},|\mathbf{\hat{v}}|}{\min}&\hspace{-6mm} \sum_{i\in\mathcal{M}_p}\frac{\left(\hat{p}_i-\tilde{p}_i\right)^2 }{2\sigma^2_{i,p}}+\hspace{-2mm}\sum_{i\in\mathcal{M}_q}\frac{\left(\hat{q}_i-\tilde{q}_i \right)^2}{2\sigma^2_{i,q}}+\hspace{-2mm}\sum_{i\in\mathcal{M}_v}\frac{\left(|\hat{v}_i|-|\tilde{v}_i|\right)^2}{2\sigma^2_{i,v}}, \nonumber\\
& \textrm{s.t.} & |\hat{\mathbf{v}}|=\mathbf{A}\hat{\mathbf{p}}+\mathbf{B}\hat{\mathbf{q}}+ |\mathbf{v}_0|,
\end{eqnarray}
where $\tilde{p}_i$, $\tilde{q}_i$ and $|\tilde{v}_i|$ represent the active, reactive and voltage magnitude measurement and vectors $\mathbf{\hat{p}}:=[\hat{p}_1,\dots,\hat{p}_N]^\intercal$, $\mathbf{\hat{q}}:=[\hat{q}_1,\dots,\hat{q}_N]^\intercal$ and $|\mathbf{\hat{v}}|:=[|\hat{v}_1|,\dots,|\hat{v}_N|]^\intercal$ collect estimation results. All the (pseudo) measurable sets for active power, reactive power and voltage magnitude $\mathcal{M}_p$, $\mathcal{M}_q$ and $\mathcal{M}_v$ are the subset of nodes $\mathcal{M}_p$, $\mathcal{M}_q$ and $\mathcal{M}_v \subseteq \mathcal{N}$. The summation of WLS estimators is normalized by their standard deviations $\sigma_{p_i}$, $\sigma_{q_i}$ and $\sigma_{v_i}$ of measurement errors with respect to $\tilde{p}_i$, $\tilde{q}_i$ and $|\tilde{v}_i|$. To ensure the full observation, we have $\mathcal{M}_p = \mathcal{M}_q = \mathcal{N}$ by using historic data as pseudo measurement with large errors.

\bibliographystyle{ieeetr}  
\bibliography{refs} 

\end{document}